\documentclass[12pt]{amsart}
\usepackage{amsmath}
\usepackage{amssymb}
\usepackage{epsfig}
\usepackage{graphicx}
\usepackage{bm}
\numberwithin{equation}{section}

\input xy
\xyoption{all}

\calclayout
\allowdisplaybreaks[3]

\theoremstyle{plain}
\newtheorem{prop}{Proposition}

\newtheorem{theo}[prop]{Theorem}

\newtheorem{lemm}[prop]{Lemma}

\theoremstyle{definition}
\newtheorem{defi}[prop]{Definition}

\newtheorem{rema}[prop]{Remark}

\newtheorem{exam}[prop]{Example}

\def\lra{\longrightarrow}
\def\ra{\rightarrow}

\def\bA{{\mathbb A}}

\def\bZ{{\mathbb Z}}
\def\bR{{\mathbb R}}

\def\lcm{\mathrm{lcm}}

\def\codim{\mathrm{codim}}

\def\Conj{\mathrm{Conj}}
\def\Burn{\mathrm{Burn}}
\def\Bir{\mathrm{Bir}}

\def\lim{\mathrm{lim}}

\def\ch{\mathrm{char}}
\def\Spec{\mathrm{Spec}}

\makeatother
\makeatletter

\author{Maxim Kontsevich}
\address{Institut des Hautes \'Etudes Scientifiques, 
35 route de Chartres, 91440 Bures-sur-Yvette, France}
\email{maxim@ihes.fr}

\author{Yuri Tschinkel}
\address{Courant Institute\\
                New York University \\
                New York, NY 10012 \\
                USA }
\email{tschinkel@cims.nyu.edu}

\address{Simons Foundation\\
160 Fifth Avenue\\
New York, NY 10010\\
USA}

\title[Specialization of birational types]{Specialization of birational types}

\begin{document}
\date{August 14, 2017}

\maketitle

\section{Introduction}

This paper is inspired by the discovery by Larsen and Lunts \cite{LL} 
of a remarkable connection between motivic
integration and stable rationality 
and by the recent development of these ideas by 
Nicaise and Shinder \cite{NS}, who proved
that stable rationality is preserved under specializations in smooth families. 
Our goal here is to simplify and strengthen their arguments, 
leading to the following: 

\begin{theo}
\label{theo:main}
Let 
$$
\pi: \mathcal X\ra B\quad\text{and}\quad \pi': \mathcal X'\ra B
$$
be smooth proper morphisms to a smooth, connected curve $B$, 
over a field of characteristic zero.
Assume that the generic fibers of $\pi$ and $\pi'$ 
are birational over the function field of $B$.
Then, for every closed point $b\in B$,
the fibers of $\pi$ and $\pi'$ over $b$ are birational over the residue field at $b$. 

In particular, if the generic fiber of $\pi$ is rational 
then every fiber of $\pi$ is rational. 
\end{theo}

Specialization of rationality was known for families of relative dimension 3 \cite{tim}, \cite{defernex}. 
In fact, in Section~\ref{sect:abs} 
we prove a stronger specialization result, when the special fiber has singularities 
of a certain type, e.g., rational double points.

Throughout, we work over a field $k$ of characteristic zero. 
In our approach, we replace the Grothendieck ring $\mathrm{K}_0(\mathrm{Var}_k)$ of  
varieties over $k$, 
which plays a key role in motivic integration 
and which is crucial in \cite{NS}, by a new invariant, the {\em Burnside ring} of $k$, denoted by 
$$
\Burn(k).
$$ 
Its definition and formal properties 
resemble those of the classical notion of the Burnside ring of a finite group (see Section~\ref{sect:burn}). 

This ring seems better adapted to 
questions of rationality: while  $\mathrm{K}_0(\mathrm{Var}_k)$, and its quotient modulo the ideal
generated by the class of the affine line, 
are natural in the context of {\em stable} rationality, the ring 
$\Burn(k)$ captures directly birational types: as an abelian group, $\Burn(k)$ is freely 
generated by isomorphism classes of function fields over $k$. 
It is naturally graded, by the transcendence degree, and admits a  surjection
$$
\Burn(k)\ra \mathrm{gr}(\mathrm{K}_0(\mathrm{Var}_k)),
$$ 
onto the associated graded, 
with respect to the dimension filtration. In the early days of motivic integration, it was believed that this map
is an isomorphism, but now it is known that the kernel is nontrivial \cite[Theorem 2.13]{borisov}.  

The proof of Theorem~\ref{theo:main} is based on the existence of a {\em specialization morphism} 
defined in Section~\ref{sect:spec}. This specialization morphism is additive but does not preserve 
the multiplicative structure of Burnside rings.
In Section~\ref{sect:abs}, we introduce Burnside {\em groups} for schemes and 
define a class of singularities relevant for rationality considerations. 
We then prove a specialization of rationality theorem in this context. 
In Section~\ref{sect:next}, we refine the notion of Burnside rings which insures the multiplicativity of
the specialization map in this more general situation.


\

\noindent
{\bf Acknowledgments:}  The second author was partially supported by NSF grant 1601912.

\section{Burnside rings of fields}
\label{sect:burn}

We recall the following classical notion. 
Let $G$ be a finite group. 
Denote by $\Burn_+(G)$ the set  
of isomorphism classes of finite $G$-sets; it 
is a commutative semi-ring
with addition and multiplication given by disjoint union and Cartesian product. 
As an additive monoid, $\Burn_+(G)$ is a free commutative monoid generated by 
the set $\Conj(G)$ of conjugacy classes of subgroups of $G$. The Burnside ring
$$
\Burn(G)
$$ 
is defined as the associated Grothendieck ring. 
It is a free $\bZ$-module generated by $\Conj(G)$. The ring $\Burn(G)$ is 
a combinatorial avatar of the representation ring $R_F(G)$, which is the 
Grothendieck ring of finite dimensional $G$-representations over a field $F$: 
there is a canonical ring homomorphism
$$
\Burn(G)\ra R_F(G)
$$
for any field $F$. 

\

This definition generalizes to pro-finite groups. In particular, 
we can apply this to Galois groups $G_k$ of fields $k$. 
In this case, the set $\Conj(G_k)$ is the set of isomorphism classes of finite field extensions $L/k$,
which we can interpret as the set of
isomorphism classes of zero-dimension $k$-schemes which are non-empty, connected, and reduced. 
The corresponding semi-ring $\Burn_+(G_k)$ can be described as the set of equivalence classes of all 
smooth zero-dimensional (not necessarily connected) schemes over $k$.  

We propose the following generalization of these notions: Let $\sim_k$ be the equivalence relation on smooth,
but possibly not connected and not equidimensional, schemes over $k$ generated by open embeddings with dense image.

\begin{defi}
\label{defi:main}
The Burnside semi-ring  $\Burn_+(k)$ of a field $k$ is the
set of $\sim_k$-equivalence classes of smooth schemes over $k$ of finite type
endowed with a semi-ring structure where multiplication and addition 
are given by disjoint union and product over $k$. 
\end{defi}

For a smooth scheme $S/k$ we denote by $[S/k]$ the corresponding $\sim_k$-equivalence class. 
As an additive monoid, $\Burn_+(k)$ 
is freely generated by
the set of $\sim_k$-equivalence classes of smooth nonempty connected schemes; for two such schemes $X,X'$ we have
$[X/k]=[X'/k]$ if and only if $X$ and $X'$ are $k$-birational. Therefore, we can identify 
additive generators of $\Burn_+(k)$ with  
$\sqcup_{n\ge 0}\Bir_n(k)$, 
where
$$
\Bir_n(k), \quad n\in \bZ_{\ge 0}
$$ 
is the set of equivalence classes of irreducible algebraic 
varieties over $k$ of dimension $n$, 
modulo $k$-birational equivalence. Abusing notation, we will write 
$[L/k]$ the class of any smooth variety $X$ with function field $L=k(X)$.  
Note that 
$$
\Bir_0(k)=\Conj(G_k).
$$ 
We denote by 
$$
\Burn(k)
$$ 
the Grothendieck ring generated by $\Burn_+(k)$. 
Clearly, $\Burn(k)$ carries a natural grading by the transcendence degree over $k$, 
$$
\Burn(k)=\oplus_{n\ge 0}\Burn_n(k)
$$ 
and   
$$
\Burn_0(k)=\Burn(G_k). 
$$

\begin{rema}
\label{rema:funct}
This construction provides a 
functor from the category of fields of characteristic zero, with morphisms given by inclusion of fields, to 
the category of $\bZ_{\ge 0}$-graded commutative rings.
\end{rema}

\section{Specialization}
\label{sect:spec}

Let $\mathfrak o$ be a complete discrete valuation ring with residue field $k$ and fraction field $K$.
Since $\ch(k)=0$, by our assumptions, we have noncanonical isomorphisms
$\mathfrak o\simeq k[[t]]$ and $K\simeq k((t))$. Our goal is to define a specialization 
homomorphism of graded {\em abelian groups} (preserving addition, but not multiplication):
\begin{equation}
\label{eqn:rho}
\rho: \Burn(K)\ra \Burn(k).
\end{equation}
This requires a collection of maps
$$
\rho_n:\Bir_n(K)\ra \bZ[\Bir_n(k)], 
$$
for all $n\in \bZ_{\ge 0}$. Here  $\bZ[\Bir_n(k)]$ is the free abelian group, generated by classes in $\Bir_n(k)$.
From these, we will obtain $\rho$ by 
$\bZ$-linearity. 

To define
$\rho_n([L/K])$ we proceed in two steps:
\begin{enumerate}
\item Choose a smooth proper (or projective) 
variety $X/K$ of dimension $n$ with function field $L=K(X)$. 
\item Choose a regular model of $X$ 
$$
\pi: \mathcal X\ra \mathrm{Spec}(\mathfrak o)
$$ 
such that $\pi$ is proper and the special fiber $\mathcal X_0$ over $\mathrm{Spec}(k)$ is 
a simple normal crossings (snc) divisor.
\end{enumerate}
The last condition means that
$$
\mathcal X_0=\cup_{i\in I} d_iD_i,
$$
where $I$ is a finite set, $d_i\in \mathbb Z_{\ge 1}$ and $D_i$ 
are smooth irreducible divisors in $\mathcal X$, with transversal intersections.  
We put
\begin{equation}
\label{eqn:formula}
\rho_n([L/K]):=\sum_{\emptyset \neq J\subseteq I}  (-1)^{\# J-1} [D_{J}\times \bA^{\#J-1}/k], 
\end{equation}
where  
\begin{equation}
\label{eqn:mult}
D_J:=\cap_{j\in J}D_j. 
\end{equation}
This formula is inspired by the formula (3.2.2) in \cite{NS}. 
Note that this map $\rho$ is a {\em not} a ring homomorphism; a modified version which {\em is} a 
ring homomorphism will be described in Section~\ref{sect:next}.

\begin{theo}
\label{theo:com}
The maps
$$
\rho_n: \Bir_n(K)\ra \bZ[\Bir_n(k)], \quad n\ge 0, 
$$
given by  \eqref{eqn:formula},
are well-defined.
\end{theo} 

\begin{proof}
We need to establish the following:
\begin{enumerate}
\item 
For a given smooth proper $X$ over $K$, 
the right side of \eqref{eqn:formula} does not depend on the choice of a model 
$\pi: \mathcal X\ra \mathrm{Spec}(\mathfrak o)$. 
\item 
Assuming this independence of the model $\mathcal X$, the right side of \eqref{eqn:formula} does not depend on the choice of 
a smooth proper model $X/K$ of the field extension $L/K$.
\end{enumerate}

The proofs of properties (1) and (2) are similar to the arguments in \cite[Section 3]{bittner}. 
We start with (1). 
By the Weak Factorization Theorem for birational maps between smooth proper $k$-varieties \cite{wlo}, \cite{abr}, 
it suffices to check that the right side of \eqref{eqn:formula} 
does not change under the following elementary transformation: the blowup
$$
\beta: \tilde{\mathcal X}:=\mathrm{Bl}_Z(\mathcal X),
$$
where $Z$ is a smooth closed irreducible subvariety in $D_{J_0}$, for some $\emptyset\neq J_0\subseteq I$, and such that 
\begin{itemize}
\item $\dim(Z)\le \dim(X)-2$,
\item $Z$ intersects the divisors $D_i$ transversally and
\item for $i\notin J_0$,
the divisors $Z_i:=D_i\cap Z$ of $Z$ 
are all distinct and form a normal crossings divisor in $Z$. 
\end{itemize}
The $k$-irreducible components of the new 
special fiber $\tilde{\mathcal X}_0$ are labeled by $\tilde{I}:=I\sqcup \{ i_Z\}$, 
where the new component corresponding to $i_Z$ is the exceptional 
divisor of $\beta$. We let $\iota:I\hookrightarrow \tilde{I}$ be the natural inclusion map.
There are two cases: 
\begin{itemize}
\item[(a)]
$\dim(Z) < \dim(D_{J_0})$, 
\item[(b)]
$Z$ is a component of $D_{J_0}$.
\end{itemize} 

We start with (a). For $J\subseteq I$, the stratum 
$D_{\iota(J)}\subset \tilde{\mathcal X}_0$ is a blowup of $D_J$,  
and thus birationally equivalent to $D_J$. 
Hence, we already matched the terms of formula \eqref{eqn:formula}, whose
indices do not contain $i_Z$.

The terms which do contain $i_Z$ are labelled by the following data: 
\begin{itemize}
\item a subset $J_1\subseteq I\setminus J_0$,   
\item an element $\alpha\in \pi_0(Z\cap D_{J_1})$,
(the set of connected components).
\item a subset $J_0'\subseteq J_0$.
\end{itemize}
The corresponding subset  in $\tilde{I}$ is   
$\tilde{J}:=\iota(J_0'\cup J_1)\cup \{ i_Z\}$. Note that
$$
\pi_0(D_{\tilde{J}})= \pi_0(Z\cap D_{J_1})
$$
 and thus we can use $\alpha\in  \pi_0(Z\cap D_{J_1})$ as a label for a component  
$D_{\tilde{J},\alpha}$ of $D_{\tilde{J}}$. Moreover, we have a $k$-birational equivalence
$$ 
D_{\tilde{J},\alpha}\times \bA^{\codim_{\tilde{\mathcal X}}(D_{\tilde{J},\alpha})-1} \sim_k 
(Z\cap D_{J_1})_{\alpha}\times \bA^{\codim_{\mathcal X}(Z\cap D_{J_1})_{\alpha}-1}.
$$

Observe that, for given $J_1$ and $\alpha$, the alternating sum over $J_0'$ vanishes. 
It follows that the sum over
all terms containing $i_Z$ is zero. This proves the Case (a) of  (1).

We turn to Case (b). Then, by assumptions, $\#J_0\ge 2$. 
For simplicity of the exposition, we may assume that $D_{J_0}$ is connected and hence, $Z=D_{J_0}$. 
The nonempty intersections of divisors in $\tilde{X}$ are of the form
\begin{itemize}
\item $D_{\iota(J)}\sim_k D_J$, for $\emptyset \neq J\subseteq I, I_0\not\subseteq J$,
\item $D_{\iota(J)\cup \{ i_Z\}} \sim_k D_{J\cup I_0} \times \bA^{\#(J\cup I_0)-\#J-1}$, for $J$ such that $I_0\not\subseteq J$.
\end{itemize}
A direct calculation shows that the right side of \eqref{eqn:formula} does not change.

To show (2), we use the Weak Factorization Theorem to reduce the claim to the case of   
a blowup 
$$
\beta: \tilde{X}:=\mathrm{Bl}_Y(X)\ra X
$$ 
with smooth center $Y\subset X$. 
Then, using embedded resolution of singularities compatible with a divisor,  
we can find a model $\pi:\mathcal X\ra \Spec(\mathfrak o)$ with special fiber
a divisor with simple normal crossings $\cup_{i\in I} d_iD_i$   
such that the closure $\mathcal{Y}$ of $Y$ in $\mathcal X$ is smooth and $D_i\cap \mathcal Y$ 
are simple normal crossings
divisors in $\mathcal Y$ (see \, e.g., \cite[Section 3]{bittner}). 
Note that the set of irreducible components of the special fiber $\mathcal X_0$ 
does not change under $\beta$ and the 
corresponding components $D_{J,\alpha}$ are replaced by their proper transforms in the blowup; 
this preserves their 
birational equivalence type. 
This proves (2).  \end{proof}

We are now in the position to deduce Theorem~\ref{theo:main}.

\begin{proof}
Let $\pi: \mathcal X\ra B$ be a smooth proper morphism to a smooth connected curve $B$ over $k$ with 
fiber $X$ over the generic point of $B$. 
Let $K=k(B)$ be the function field of $B$. 
Let $\kappa_b$ be the residue field at $b$, a finite extension of $k$. Let 
$K_b$ be the completion of $K$ at $b$. It is a local field with residue field $\kappa_b$, isomorphic to 
$\kappa_b((t))$, where $t$ is a formal local coordinate. Let 
$$
\phi_b  : K\ra K_b
$$
be the canonical inclusion.  By functoriality (see Remark~\ref{rema:funct}), it defines a homomorpism
$$
\phi_{b,*} : \Burn(K)\ra \Burn(K_b).
$$ 
We have the specialization homomorphism 
$$
\rho: \Burn(K_b)\ra \Burn(\kappa_b)
$$
and the following identity
$$
[X_b/\kappa_b] = \rho(\phi_{b,*}([X/K])),
$$ 
which follows immediately from the definition of $\rho$, since the special fiber smooth and irreducible.
This shows that the birational type of the fiber is determined by the birational type at the generic point. 
\end{proof}

\begin{rema}
A variety $X/k$ is called stably rational of level $\le r$ if $X\times \mathbb A^r$ is $k$-rational. 
An immediate corollary is that if 
the generic fiber of a family $\pi: \mathcal X\ra B$, as in Theorem~\ref{theo:main}, 
is stably rational of level $\le r$ then 
every fiber is stably rational of level $\le r$. 
\end{rema}

\begin{rema}
In the notation of Theorem~\ref{theo:main}, given a birational isomorphism of generic fibers of 
$\pi$ and $\pi'$, 
the corresponding birational isomorphism of the special fibers is not determined uniquely. 
However, following the proof of Theorem~\ref{theo:com},
we can determine an explicit chain of elementary birational 
modifications relating the special fibers
$\mathcal X_{b}$ and $\mathcal X'_b$, for a given $b\in B$.   
\end{rema}

\begin{defi}
A variety $X/k$ is called (birationally) $\bA^1$-divisible, if there exists a variety $Y/k$ 
such that $[X/k]=[Y\times \bA^1/k]$.
\end{defi}

\begin{rema}
In the notation of Theorem~\ref{theo:main}, if the generic fiber of $\pi$ is $\bA^1$-divisible
then so is the special fiber. 
\end{rema}

\begin{rema}
\label{rema:bira}
Assume that for some $b\in B$, the fiber $\mathcal X_b$ is not $\bA^1$-divisible. 
The considerations above show that there is a canonical restriction 
homomorphism between groups of birational 
automorphisms 
\begin{equation}
\label{eqn:bira}
\mathrm{BirAut}(X/K)\ra \mathrm{BirAut}(\mathcal X_b/\kappa_b).
\end{equation}
Indeed, using formula \eqref{eqn:rho} and 
following the steps of the proof of Theorem~\ref{theo:com}, 
we see that for every model $\pi:\mathcal X\ra B$ 
of $L=K(X)$ and every local model 
$\mathcal X'_b\ra \Spec(\mathfrak o)$,  
every summand, but one, in the formula \eqref{eqn:formula} 
for the image of the specialization map $\rho$,
is $\bA^1$-divisible. This allows to define the homomorphism
\eqref{eqn:bira}.
\end{rema}

\section{Burnside groups for schemes}
\label{sect:abs}

\begin{defi}
\label{defi:new}
Let $S/k$ be a separated scheme of finite type over $k$. 
The Burnside monoid $\Burn_+(S/k)$ is the set of equivalence 
classes of maps $f: X\ra S$, where $X$ smooth over $k$, modulo the equivalence relation 
generated by 
$$
(X,f)\sim_k (U,f|_U), 
$$
where $U\hookrightarrow X$ is an open embedding with dense image. 
The monoid structure on $\Burn_+(S/k)$ is given by disjoint union. 
\end{defi}

We write $[X\stackrel{f}{\lra} S]$ for the equivalence class defined above. We denote by 
$$
\Burn(S/k)
$$ 
the corresponding Grothendieck group (it is no longer a ring). As before, the Burnside monoid and group are 
naturally graded by the dimension of $X$ over $k$.   
We recover Definition~\ref{defi:main} when  $S=\Spec(k)$.

Note that $\Burn_{+,n}(S/k)$, for $n\in \bZ_{\ge 0}$, 
is freely generated by the set 
$$
\coprod_{s\in S, \,\dim(s)\le n} \Bir_{n-\dim(s)}(\kappa_s),
$$
where $\kappa_s$ is the residue field of the Zariski point $s\in S$. 

\

\noindent
{\bf Functoriality:}
A morphism of schemes of finite type $g:S'\ra S$ over $k$ induces a homomorphism of Burnside groups
$$
g_*:\Burn(S'/k)\ra \Burn(S/k),  \quad g_*([X\xrightarrow{f} S']):=[X\xrightarrow{g\circ f} S],
$$
preserving the grading.

Let $X/k$ be an algebraic variety,
possibly reducible, nonproper, or singular. 
Let $Z\subset X$
be a closed subvariety of dimension $<\dim(X)$, 
containing the singular locus $X^{\mathrm sing}$ of $X$, 
so that $X\setminus Z$ is smooth. To such a pair $(X,Z)$ we will associate an element
$$
\partial_Z(X)\in \Burn_{\dim(X)-1}(Z/k).
$$  
These assigments will satisfy the following conditions
\begin{enumerate}
\item 
If $X$ is smooth and $Z=\cup_{i\in I} D_i$ is an snc divisor in $X$ then
\begin{equation}
\label{eqn:0}
  \partial_Z(X) =\sum_{\emptyset \neq J\subseteq I} (-1)^{\#J-1} [D_{J}\times \bA^{\#J-1}\stackrel{f_J}{\lra} Z],
\end{equation}
where $f_J$ is the composition of projection to the first factor with the natural inclusion $D_J\hookrightarrow Z$. 
\item 
Let $g:X'\ra X$ be a proper surjective morphism and $Z':=g^{-1}(Z)$. Assume that $(X',Z')$ 
satisfies the conditions above and that 
$g$ induces an isomorphisms
$$
X'\setminus Z'\ra X\setminus Z,
$$ 
in particular, $\dim(X')=\dim(X)$.
Then 
\begin{equation}
\label{eqn:01}
\partial_{Z}(X)=(g|_Z)_*(\partial_{Z'}(X')).
\end{equation}
\end{enumerate}

\begin{theo}
\label{theo:abs}
The invariants $\partial_Z(X)$
satisfying (1) and (2) exist and are uniquely defined. 
\end{theo}

\begin{proof}
The properties (1) and (2) provide a definition, using resolution of singularities. The proof of 
independence on the choice of such a resolution is parallel to the proof of property (1) in 
Theorem~\ref{theo:com}. 
\end{proof}

\begin{defi}
\label{defi:key}
Let $X/k$ be an irreducible algebraic variety and 
$Z\subset X$ an irreducible divisor containing the singular locus
$X^{\mathrm{sing}}$ of $X$. The pair $(X,Z)$ has {\em B-rational singularities} 
if 
$$
\partial_Z(X)=[Z^{\mathrm{smooth}}\hookrightarrow Z].
$$
\end{defi}

Note that the difference
$$
\partial_Z(X)-[Z^{\mathrm{smooth}}\hookrightarrow Z]
$$
is always supported in 
$Z^{\mathrm{sing}}\cup X^{\mathrm{sing}}$. In particular, if 
$Z$ and $X$ are both smooth, then the pair $(X,Z)$ trivially has B-rational singularities. 

\begin{exam}
Let $X/k$ be a smooth variety and $Z\subset X$ a divisor, defined over $k$, 
with an isolated singularity at $z\in Z(k)$. 
Let $\beta:\tilde{X}:=\mathrm{Bl}_z(X)$ be the blowup. 
Assume that
\begin{itemize}
\item the proper transform $Z'$ of $Z$ 
is smooth and intersects
the exceptional divisor $E\simeq \mathbb P^{\dim(X)-1}$ transversally,
\item the cone over $E\cap Z'$ is $k$-rational.
\end{itemize} 
Then the pair $(X,Z)$ has B-rational singularities.

Indeed, on $\tilde{X}$ we have the divisor 
$\tilde{Z}:=\beta^{-1}(Z)= Z\cup E$. 
The formula \eqref{eqn:0} 
for $\partial_{\tilde{Z}}(\tilde{X})$ has three terms:
$$
[E\ra \tilde{Z}] + [Z'\ra \tilde{Z}] - [(E\cap Z')\times \bA^1 \ra \tilde{Z}]
$$    
We apply \eqref{eqn:01} to the blowup $\beta$ and find 
$$
\partial_Z(X):=[ \bA^{\dim(X)-1 }\ra z] + [Z^{\mathrm{smooth}}\hookrightarrow Z]  - [ (E\cap Z')\times \bA^1\ra z].
$$
By our assumption, we find that the first and the third terms cancel. Hence
$$
\partial_Z(X):= [Z^{\mathrm{smooth}}\hookrightarrow Z],
$$
which is the definition of B-rationality. 
\end{exam}

\begin{exam}
Consider pairs $(X,Z)$ with smooth $X/k$. Assume that $Z$ is a divisor that has only 
ordinary double point singularities at $k$-rational points, and 
such that the associated projective quadrics all have $k$-rational points. 
Then the pair $(X,Z)$ has B-rational singularities. 
\end{exam}

\begin{exam}
Consider pairs $(X,Z)$, with $X/k$ a smooth surface (e.g., $\bA^2$). Assume that 
$Z\subset X$ is a closed irreducible curve over $k$ 
which is unibranched at every singular point. 
Then $(X,Z)$ has B-rational singularities.  
\end{exam}

\begin{theo}
\label{theo:main2}
Let $\pi:\mathcal X\ra B$ be a proper flat morphism to a smooth connected curve $B$ over $k$. 
Assume that the generic fiber $X$ is a smooth geometrically connected variety over $K:=k(B)$. 
Let $b\subset B$ be a closed point and $\mathcal X_b:=\pi^{-1}(b)$ the fiber over $b$. Assume that
the pair $(\mathcal X, \mathcal X_b)$ has B-rational singularities. 
If $X$ is rational over $K$ then $\mathcal X_b$ is rational over the residue field $\kappa_b$ at $b$. 
\end{theo}

\begin{proof}
Let $\beta: \tilde{\mathcal X}\ra \mathcal X$ be a resolution of singularities of $\mathcal X$
which is an isomorphism on the complement to 
$\mathcal X_b\cup \mathcal X^{\mathrm{sing}}$ and such that
 $\beta^{-1}(\mathcal X_b \cup \mathcal X^{\mathrm{sing}})$ is
an snc divisor in $\tilde{\mathcal X}$. By our assumption that the generic fiber of $\pi$ is smooth, 
we have that
$\tilde{\mathcal X}_b=\beta^{-1}(\mathcal X_b)$ is a connected (but possibly reducible) component of 
$\beta^{-1}(\mathcal X_b \cup \mathcal X^{\mathrm{sing}})$. 
We have
$$
\beta_*(\partial_{\tilde{\mathcal X}_b}(\tilde{\mathcal X})) = \partial_{\mathcal X_b}(\mathcal X)
$$
By the assumption that the singularities of the pair $(\mathcal X, {\mathcal X}_b)$ are B-rational, we also have
\begin{equation}
\label{eqn:1}
\partial_{\mathcal X_b}(\mathcal X)= [{\mathcal X}_b^{\mathrm{smooth}}\hookrightarrow  \mathcal X_b].
\end{equation}
We have the natural projections
\begin{equation}
\Pi_b:=\pi|_{\mathcal X_b}\ra b, \quad    \tilde{\Pi}_b:\tilde{\pi}|_{\tilde{\mathcal X}_b}\ra b,
\label{eqn:2}
\end{equation}
(we identify $b=\Spec(\kappa_b)$), and  
$$
\tilde{\Pi}_b = \Pi_b\circ \beta|_{\tilde{\mathcal X}_b}, \quad \tilde{\pi}= \pi\circ \beta. 
$$
Apply the induced homomorphism $(\Pi_b)_*$ to both sides of \eqref{eqn:1}. Using functoriality, we obtain that 
\begin{equation}
\label{eqn:3}
(\tilde{\Pi}_b)_*(\partial_{\tilde{\mathcal X}_b}(\tilde{\mathcal X}))= [  \mathcal X_b^{\mathrm{smooth}} \ra b] 
\end{equation}
Note that both sides of \eqref{eqn:3} are elements of 
$\Burn(b/\kappa_b)= \Burn(\kappa_b)$. 
The left side of \eqref{eqn:3} is, by definition, 
$$
\rho(\phi_{b,*}[X/K]),
$$
where $\rho$ is the specialization homomorphism defined in Section~\ref{sect:spec} and $\phi_b: K\ra K_b$ is the natural embedding
into the completion at $b$. 

Rationality of $X$ over $K$ implies that  $[X/K]=[\bA^{\dim(X)}/K]$, so that 
the left side of \eqref{eqn:3} equals 
$$
[\bA^{\dim(X)} \ra b] = [\bA^{\dim(X)}/\kappa_b].
$$ 
Hence $\mathcal X_b$ is rational. 
\end{proof}

\begin{rema}
Recently, there has been rapid progress in establishing failure of (stable) rationality via specialization
to varieties with computable obstructions to (stable) rationality, see, e.g., 
\cite{voisin}, \cite{ctp}, \cite{totaro}, \cite{HKT}. Previously, specialization was
studied on the level of integral decomposition of the diagonal \cite{voisin}, 
universal $\mathrm{CH}_0$-triviality \cite{ctp}, 
or $\mathrm{K_0}(\mathrm{Var}_k)$ \cite{NS}. We expect 
that the flexibility of Theorem~\ref{theo:main2} 
will lead to new applications in this area.  
\end{rema}

\section{Refined specialization}
\label{sect:next}

We keep the notation of Section~\ref{sect:spec}: we consider a regular model
$$
\pi:\mathcal X\ra \Spec(\mathfrak o),
$$
with special fiber 
$$
\mathcal X_0:=\cup_{i\in I} d_iD_i.
$$ 
We defined an additive homomorphism 
$$
\rho:\Burn(K)\ra \Burn(k),
$$
where $K$ is the fraction field of $\mathfrak o$ 
and $k$ its residue field. Our first goal here is to decompose 
$$
\rho:=\sum_{d \ge 1} \rho^{(d)},
$$
where 
$$
\rho^{(d)}:=\sum_{n\ge 0} \rho^{(d)}_n, \quad  \rho^{(d)}_n: \Bir_n(K)\ra \bZ[\Bir_n(k)].
$$
The presentation as an infinite sum is meaningful, since for any given element in $\Burn(K)$, 
the evaluation of $\rho^{(d)}$ vanishes for sufficiently large $d$. 
Put
\begin{equation}
\label{eqn:r1}
\rho_n^{(d)}([L/K]):=\sum_{\emptyset \neq J\subseteq I, d_J=d}  (-1)^{\# J-1} [D_{J}\times \bA^{\#J-1}/k], 
\end{equation}
where  
$$
d_J:=\gcd(d_j)_{j\in J}.
$$
The proof that $\rho_n^{(d)}$ is well-defined is identical to the proof of Theorem~\ref{theo:com}, one has only 
to keep track of the multiplicities $d_J$ on the blowup. 

We can organize the additive homomorphisms into a generating series
\begin{equation}
\label{eqn:gen}
\hat{\rho}:=\sum_{d\ge 1} \rho^{(d)} T^d : \Burn(K)\ra \Burn(k)[T],
\end{equation}
where $T$ is a formal variable; evaluating at $T=1$ we obtain $\rho$. 

Our ultimate goal is to define a specialization homomorphism 
$$
\rho^{\bm{\mu}}_{\tau}: \Burn(K)\ra \Burn^{\hat{\bm{\mu}}}(k), 
$$
to an equivariant version of the Burnside ring defined below. The homomorphism $\rho^{\bm{\mu}}_{\tau}$ 
will be  compatible with multiplication. 

\

{\em Step 1.}
Let
$$
\hat{\bm{\mu}}:=\varprojlim \bm{\mu}_d,
$$
where $\bm{\mu}_d$ are $k$-group schemes of $d$-th roots of 1. We define rings
$$
\Burn^{\hat{\bm{\mu}}}(k)
$$
in same manner as $\Burn(k)$, except that we now consider smooth $k$-schemes, together with a continuous $\hat{\bm{\mu}}$-action. 
The multiplication is given by the product of schemes.
The additive generators of this Grothendieck group are equivalence classes 
$[\bar{X}\ra X,d]$, where $d\in \bZ_{\ge 1}$, and $\bar{X}\ra X$ is 
a $\bm{\mu}_d$-torsor over a smooth scheme $X$. In terms of fields, the generators are isomorphism classes 
$[L/k,\lambda,d]$, where
$L/k$ is a function field, $d\in \bZ_{\ge 1}$, and $\lambda\in L^\times/(L^\times)^d$ (by Kummer theory). 

\

{\em Step 2.}
Fix a nonzero element 
$$
\tau\in \mathfrak m/\mathfrak m^2\simeq k,
$$ 
where $\mathfrak m$ is the maximal ideal of $\mathfrak o$. 
The homomorphism $\rho^{\bm{\mu}}_{\tau}$ depends on this choice.

\

{\em Step 3.}
As in Section~\ref{sect:spec}, we assume that the $k$-irreducible components $D_i$ of the special fiber
$\mathcal X_0$ form a strict normal crossings divisor. Denote by $\mathcal N_i$ the normal bundle of $D_i$. The fibration
structure $\pi:\mathcal X\ra \Spec(\mathfrak o)$ defines, for all $\emptyset \neq J\subseteq I$, an isomorphism 
\begin{equation}
\label{eqn:tauu}
\nu_J':\bigotimes_{j\in J} \left({{\mathcal N}_j}_{|_{D_J^\circ}}\right)^{\otimes d_j} \stackrel{\sim}{\longrightarrow} \mathfrak m/\mathfrak m^2\otimes_k 
\mathcal O_{D_J^\circ},
\end{equation}
where 
$$
D_J^\circ=D_J \setminus \cup_{i\notin J} D_i.
$$
Using the element $\tau$ we identify the right side of \eqref{eqn:tauu} with the trivial line bundle and obtain an isomorphism
$$
\nu_J: \mathcal L_J^{\otimes d_J}  \stackrel{\sim}{\longrightarrow} \mathcal O_{D_J^\circ},
$$
where 
$$
\mathcal L_J:=\bigotimes_{j\in J} \left({{\mathcal N}_j}_{|_{D_J^\circ}}\right)^{\otimes d_j/d_J}.
$$
 
Define $\bar{D}^\circ_J$ as the locus of points in the total space of the line bundle $\mathcal L_J$, whose $d_J$-th tensor power
is mapped to 1 by $\nu_J$. By construction, $\bar{D}^\circ_J\ra D_J^\circ$ is a $\bm{\mu}_{d_J}$-torsor.

\

{\em Step 4.}
Define
\begin{equation}
\label{eqn:uu}
\rho_{\tau}^{\bm{\mu}} ([X/K]) :=\sum_{\emptyset\neq J\subseteq I} (-1)^{\#J-1} [\bar{D}_J^\circ\times \bA^{\#J-1}\ra  D_J^\circ\times \bA^{\#J-1},d_J]
\end{equation}

\begin{theo}
\label{theo:tau}
The additive homomorphism 
$$
\rho_{\tau}^{\bm{\mu}}:\Burn(K)\ra \Burn^{\hat{\bm{\mu}}}(k)
$$
is well-defined and is a ring homomorphism. 
\end{theo}

\begin{proof}
The verification that $\rho_{\tau}^{\bm{\mu}}$ is well-defined is completely analogous to the proof of Theorem~\ref{theo:com}, 
we only need to keep track of the torsor structures, which is straightforward because all intersections of components of snc
divisors on the blowup are birationally equivalent to projective bundles.

The main difficulty is the proof of multiplicativity. 
The issue we face is that the fiber product of two snc models over $\Spec(\mathfrak o)$ 
is not, in general, an snc model: it has toric singularities of a particular type. 
To resolve this issue is suffices to provide an explicit resolution of singularities of affine 
toric varieties of the following type:
$$
\prod_{j\in J}x_j^{a_j} = \prod_{j'\in J'} y_{j'}^{a_{j'}},
$$
where $J,J'$ are nonempty, and $a_j,a_{j'}\ge 1$, for all $j\in J, j'\in J'$. 
Moreover, the family of resolutions
should be compatible with restrictions, when some of the variables are set to 1. Using the toric language, 
we translate this problem to a combinatorial construction: for every pair $r,r'\in \bZ_{\ge 1}$ and vectors
$$
\mathbf{a}:=(a_1,\ldots, a_r)\in \bZ^r_{\ge 1}, \quad 
\mathbf{a}':=(a_1,\ldots, a_{r'})\in \bZ_{\ge 1}^{r'}
$$ 
consider
the cone 
$$
\Lambda(\mathbf{a}, \mathbf{a}'):=\{ (\mathbf{u},\mathbf{u}')\in  \bR^{r+r'}_{\ge 0} \, \mid \, 
\langle \mathbf{u}, \mathbf{a}\rangle =\langle\mathbf{u}', \mathbf{a}'\rangle \} \subset \bR^{r+r'}, 
$$
where $\langle \cdot, \cdot\rangle$ is the standard scalar product.

Note that all boundary strata of $\Lambda(\mathbf{a}, \mathbf{a}')$ are 
naturally isomorphic to  $\Lambda(\mathbf{b}, \mathbf{b}')$, 
for some integral vectors $\mathbf{b}, \mathbf{b}'$ of smaller
dimension (up to permutation of indices). 

Then choose fans $\Sigma(\mathbf{a}, \mathbf{a}')$
supported in  $\Lambda(\mathbf{a}, \mathbf{a}')$, for all $\mathbf{a}, \mathbf{a}'$, such that the collection 
of these fans satisfies the following properties:
\begin{itemize}
\item All faces  
$\sigma\in \Sigma(\mathbf{a}, \mathbf{a}')$ are primitive and simplicial, 
i.e., the corresponding toric variety
$X_{\Sigma(\mathbf{a}, \mathbf{a}')}$ is smooth. 
\item The choice of fans is covariant with respect to symmetric groups 
$\mathfrak S_{r}\times \mathfrak{S}_{r'}$
acting on the indices.
\item 
For all $\mathbf{a}, \mathbf{a}'$, the induced fan on a boundary of $\Lambda(\mathbf{a}, \mathbf{a}')$
coincides with the corresponding fan  $\Sigma(\mathbf{b}, \mathbf{b}')$. 
\end{itemize} 
The existence of such an equivariant, compatible with restriction to boundary, family of fans can be proved
inductively (see, e.g.,\cite{ct}).

Now consider two snc models $\mathcal X, \mathcal X'$ of smooth proper 
$K$-varieties $X$ and $X'$,
with function fields $K(X), K(X')$.  
A choice of a family of fans as above determines an snc model $\mathcal W$  over $\Spec(\mathfrak o)$
of the product $W:=X\times X'$, endowed with a natural map  
$$
\mathcal W\ra \mathcal X\times_{\Spec(\mathfrak o)} \mathcal X'.
$$
The special fiber $\mathcal W_0$ 
is stratified, with strata 
labelled by:
\begin{itemize}
\item[(i)] nonempty subsets $J\subseteq I, J'\subseteq I'$,  
\item[(ii)] $\sigma\in \Sigma(\mathbf{a}, \mathbf{a}')$ such that its interior is contained in 
the interior of $\Lambda(\mathbf{a}, \mathbf{a}')$, where the entries of the vector 
$\mathbf{a}$, respectively, $\mathbf{a}'$, are $d_j$, $j\in J$, respectively, $d_{j'}$, $j'\in J'$,
(with some enumeration of elements in $J,J'$). 
\end{itemize}
For each triple $(J,J', \sigma)$ as above, the corresponding open
stratum  $D^\circ_{J,J',\sigma}$ is a disjoint union of open snc strata in $\mathcal W_0$. Moreover, 
it is the base of a $\bm{\mu}_{\lcm(d_J,d_{J'})}$-torsor 
$$
\bar{D}^\circ_{J,J',\sigma}\ra D^\circ_{J,J',\sigma}.
$$ 

\begin{lemm}
\label{lemm:dj}
There exists a natural isomorphism of  $\bm{\mu}_{\lcm(d_J,d_{J'})}$-torsors

\

\centerline{
\xymatrix{
\bar{D}^\circ_{J,J',\sigma}\ar[d] \ar[r]^{\sim\hskip 0.7cm } &  \bar{D}^\circ_J\times \bar{D}^\circ_{J'}\times \mathbb T_\sigma  \ar[d]   \\
 D^\circ_{J,J',\sigma} \ar[r]^{\sim\hskip 2cm } &  (D^\circ_J\times D^\circ_{J'})/\bm{\mu}_{\lcm(d_J,d_{J'})}   \times \mathbb T_\sigma 
}
}

\

\noindent
where $\mathbb T_{\sigma}$ is a split torus over $k$. 
\end{lemm}

\begin{proof}
This is a straightforward verification.  
\end{proof}

Fix nonempty $J,J'$. 
We substitute \eqref{eqn:formula} into
$$
\rho_{\dim(X)}([K(X)/K])\cdot \rho_{\dim(X')}([K(X')/K])
$$
and extract products of terms
corresponding to $J$ and $J'$. We claim  that 
the sum of these terms equals to the sum over $\sigma$ 
of contributions of snc strata in 
$\mathcal W_0$, labelled by $J,J'$ and $\sigma$ as in (ii) above. 
This claim follows from the equality (Euler characteristic computation):
$$
\sum_{\sigma} (-1)^{\dim(\sigma)} = (-1)^{\dim(\Lambda(\mathbf{a}, \mathbf{a}'))},
$$
where $\sigma$ are subject to (ii). 
Multiplicativity of $\rho$ now follows by taking summation over all $J,J'$. 
\end{proof}

We now explain the relation between the multiplicative specialization homomorphism we constructed
and $\hat{\rho}$ defined in \eqref{eqn:gen}. 
Consider additive maps
$$
\Burn^{\hat{\bm{\mu}}}(k)\stackrel{\psi}{\lra} \Burn(k)[T]\xrightarrow{T=1} \Burn(k)
$$
given by
$$
[Y'\ra Y,d]\mapsto [Y/k]\cdot T^d \mapsto [Y/k].
$$
We have
$$
\hat{\rho}=\psi\circ \rho^{\bm{\mu}}_{\tau},
$$
and specializing $T=1$ we obtain $\rho$ from \eqref{eqn:rho}. 
The advantage of the nonmultiplicative homomorphisms $\rho$ and $\hat{\rho}$ is that they do not depend on 
the choice of the cotangent vector $\tau$. 

There is also a {\em ring homomorphism}
$$
\chi: \Burn^{\hat{\bm{\mu}}}(k) \ra \Burn(k)
$$
given by
$$
[Y'\ra Y,d]\mapsto [Y/k]\cdot T^d \mapsto [Y'/k].
$$
Composing $\rho^{\bm{\mu}}_{\tau}$ with this homomorphism we obtain a ring homomorphism
$$
\Burn(K)\ra \Burn(k),
$$
depending on $\tau$. 
This is the Burnside ring version of motivic reduction/motivic volume \cite[Theorem 2.6.1]{NP}, \cite{HK}. 

\begin{rema}
These finer homomorphisms can be applied to considerations in Section~\ref{sect:abs}, leading to more refined
invariants. 
\end{rema}

\bibliographystyle{alpha}
\bibliography{types}
\end{document}